\documentclass[10pt,oneside,reqno]{amsart}

\usepackage{amsmath,amssymb}
\usepackage{amsthm} 
\usepackage{array}
\usepackage{mathtools}
\usepackage{enumerate}
\usepackage{hyperref}

\theoremstyle{thmstyleone}%
\newtheorem{theorem}{Theorem}
\newtheorem{proposition}[theorem]{Proposition}%

\theoremstyle{thmstyletwo}%
\newtheorem{example}{Example}%
\newtheorem{remark}{Remark}%

\theoremstyle{thmstylethree}%
\newtheorem{definition}{Definition}%

\newtheorem{lemma}[theorem]{Lemma}
\newtheorem{conjecture}{Conjecture}

\allowdisplaybreaks
 
\def\<{\prec}
\def\>{\succ}

\def\bes{\begin{equation*}}
\def\be{\begin{equation}}
\def\ee{\end{equation}}
\def\ees{\end{equation*}}

 \def\cA{\mathcal{A}}
 \def\cD{\mathcal{D}}
 \def\cM{\mathcal{M}}
\def\ff{\varphi}
 \def\bR{\mathbf{R}}
\def\bN{\mathbf{N}}

\begin{document}
\title{Coinciding mean of the two symmetries on the set of mean functions}
\author{Lenka Mihoković}
\begin{abstract}{On the set $\mathcal M$ of mean functions the symmetric mean of $M$ with respect to mean $M_0$ can be defined in several ways. 
The first one is related to the group structure on $\mathcal M$ and the second one is defined trough Gauss' functional equation. In this paper we provide an answer to the open question formulated by B.\ Farhi about the matching of these two different mappings called symmetries on the set of mean functions. Using techniques of asymptotic expansions developed by T.\ Buri\'c, N.\ Elezovi\'c and L.\ Mihokovi\'c (Vuk\v si\'c) we discuss some properties of such symmetries trough connection with asymptotic expansions of means involved. As a result of coefficient comparison, new class of means was discovered which interpolates between harmonic, geometric and arithmetic mean.}
\end{abstract}
\maketitle

\section{Introduction}

Function $M \colon \bR^+\times \bR^+ \to \bR$ is called a mean if for all $s,t\in\bR^+$
\be\label{defM-int}
	\min(s,t) \le M(s,t) \le \max(s,t).
\ee
Mean $M$ is symmetric if for all $s,t\in\bR^+$
\bes
	M(s,t) = M(t,s)
\ees
and homogeneous (of degree 1) if for all $\lambda,s,t\in \bR^+$
\bes
	M(\lambda s, \lambda t ) =\lambda M(s,t).
\ees
This paper was motivated by the problem of matching two different mappings on the set of mean functions formulated in paper \cite{Farhi} in which author introduced algebraic and topological structures on the set $\cM_\cD$ of symmetric means on a symmetric domain $\cD$ with additional property
\bes
	M(s,t)=s \Rightarrow s=t,\quad \forall (s,t)\in\cD.
\ees
The first mapping is related to the group structure and the second one is defined trough Gauss' functional equation. It was found that those mappings coincide for arithmetic, geometric and harmonic mean but the question of the existence of other solutions remained open. 
We shall take $\cD=\bR^+\times \bR^+$.

		First, let $\cA_\cD$ be set of all functions $f\colon\cD \to \bR$ such that
  		\bes
  			(\forall (x,y)\in \cD)\ f(x,y) = -f(y,x). 
 		\ees			
 		$(\cA_\cD, +)$ is an abelian group with the neutral element $0$.	
		Function $ \ff \colon \cM_\cD \to \cA^\cD$ defined by
  		\bes
  			 \ff(M) (x,y) \coloneqq 
  			\begin{cases}
  				\log\left(-\frac{M(x,y)-x}{M(x,y)-y}\right), & x\neq y,\\
  				0, & x=y,
 			\end{cases}
 		\ees
 		is a bijection. 		
 		The composition law $\ast\colon \cM_\cD\times\cM_\cD\to\cM_\cD$ is defined by
  		\bes
  			M_1 \ast M_2 =\ff^{-1}(\ff(M_1)+\ff(M_2)).
 		\ees
 		Thus
 		$(\cM_\cD, \ast)$  is an abelian group with the neutral element $\ff^{-1}(0)=A$.
 		It can also easily be shown that the explicit formula for the composition law $\ast$ holds:	
 		\be\label{fromula-ast-explicit}
 			(M_1\ast M_2)(x,y) =\begin{cases} 
 				\frac{x(M_1-y)(M_2-y)+y(M_1-x)(M_2-x)}{(M_1-x)(M_2-x)+(M_1-y)(M_2-y)},&x\neq y,\\
 				x, &x=y.
			\end{cases}
		\ee	
		Based on this operation, the first type of the symmetry was defined.
		\begin{definition}[\cite{Farhi}]
			The symmetric mean $M_2$  to a mean $M_1$ with respect to mean $M_0$ 
			via the group structure $(\cM_\cD, \ast)$ is defined with the expression
			\be\label{defn-S}
				S_{M_0}(M_1)= M_2 \Leftrightarrow M_1\ast M_2 = M_0\ast M_0.
			\ee
		\end{definition}
		Combining \eqref{defn-S} with \eqref{fromula-ast-explicit} the explicit formula for symmetric mean of mean $M_1$ with respect to $M_0$ can easily be calculated:
 		\be\label{SM0-explicit}
 			S_{M_0}(M_1)=\frac{x(M_1-x)(M_0-y)^2-y(M_0-x)^2(M_1-y)}{(M_1-x)(M_0-y)^2-(M_0-x)^2(M_1-y)}.
		\ee	
		
We shall see the behavior of $S_{M_0}$ for some basic well known means $M_0$.
For $(s,t)\in\cD=\bR^+\times \bR^+$ let
	\bes
		A(s,t)=\frac{s+t}2,\quad
		G(s,t)=\sqrt{st},\quad
		H(s,t)=\frac{2st}{s+t},
	\ees
be the arithmetic, geometric and harmonic means respectively.
		\begin{example}[\cite{Farhi}]\label{exa-S}
		For any mean $M\in\cM_\cD$, we have:
		\begin{enumerate}
			\item $S_A(M) = 2A-M,$
			\item $S_G(M) = \frac{G^2}M$,
			\item $S_H(M) = \frac{HM}{2M-H}$.
		\end{enumerate}
	\end{example}

	Another type of symmetry, independent of the group structure  $(\cM_\cD, \ast)$, can also be defined.
	\begin{definition}[\cite{Farhi}]
	Mean $M_2$ is said to be functional symmetric mean of $M_1$ with respect to $M_0$ if the following functional equation is satisfied:
		 \be\label{GAuss}
	 		\sigma_{M_0}(M_1) =M_2 \Leftrightarrow M_0(M_1,M_2) = M_0.
		\ee
	\end{definition}
We can also say that mean $M_0$ is the functional middle of $M_1$ and $M_2$.
 Defining equation on the right side of the equivalence relation \eqref{GAuss} is known as the Gauss funcional equation.
 Some authors refer to means $M_1$ and $M_2$ as a pair of $M_0$-complementary means. Mean $M_0$ is also said to be $(M_1,M_2)$-invariant.
 For recent related results see \cite{MatkNovWitk,ToadCostToad,ToadRassToad} and also survey article on invariance of means \cite{JarJar} and references therein.
Furthermore, functional symmetric mean exists and it is unique.		

With respect to the same means as in the latter exmple we may calculate the symmetric means. 	
		\begin{example}[\cite{Farhi}]\label{exa-sigma}
		For any mean $M\in\cM_\cD$, we have:
		\begin{enumerate}
			\item $\sigma_A(M) = 2A-M,$
			\item $\sigma_G(M) = \frac{G^2}M$,
			\item $\sigma_H(M) = \frac{HM}{2M-H}$.
		\end{enumerate}
	\end{example}
	
Taking into account Examples \ref{exa-S} and \ref{exa-sigma} in which the same mappings appear with respect to arithmetic, geometric and harmonic mean appear, author  in \cite{Farhi} states the following.
	
	{\bf Open question}.
	 	For which mean functions $M_0$ on $\cD=\bR^+ \times \bR^+$ the two symmetries, $S$ and $\sigma$, with respect to $M_0$ coincide?
	 	
The goal of this paper is to analyze the open question and offer the answer in the setting of symmetric homogeneous means which possess the asymptotic expansion. Techniques of asymptotic expansions were developed in \cite{ChenElVu-2013,ElVu-09,ElVu-04} 
and appeared to be very useful in comparison and finding inequalities for bivariate means (\cite{El,ElMih-sarajevo}),
comparison of bivariate parameter means (\cite{ElVu-04}), finding optimal parameters in convex combinations of means (\cite{ElMih-sarajevo,Vu-2015}) and solving the functional equations of the form $B(A(x))=C(x)$, where asymptotic
expansions of $B$ and $C$ are known, in order to obtain the asymptotic expansion of integral means (\cite{ElVu-10}).

Techniques and results applyed in this paper were described in Section 2. 
In the next step we obtained the algorithm for calculating the coefficients in the asymptotic expansions of means
	$M_2^S=S_{M_0}(M_1)$ and $M_2^\sigma=\sigma_{M_0}(M_1)$.
Comparing the first few obtained coefficents we anticipated the general form of the coefficients in the asymptotic expansion of mean $M_0$ for which $M_2^S=M_2^\sigma$.

At the beginning of Section 3 we found closed formula and proved that proposed function represents the well defined one parameter class of means. Later we have shown that it also covers, as the special cases, means from Examples \ref{exa-S} and \ref{exa-sigma}. Other properties were also explored such as limit behavior and monotonocity with respect to the parameter.

Lastly, in Section 4 we have proved that this class of means answered the open question and stated the hypothesis that there weren't any other solutions in the context of homogeneous symmetric means wich possess asymptotic power series expansions.

In addition, methods presented in this paper may be useful with similar problems regarding functional equations, especially in case when the explicit formula for included function was not known.


\section{Asymptotic expansions}

Recall the definition of an asymptotic power series expansion as $x\to\infty$.
\begin{definition}
		The series $\sum_{n=0}^{\infty} c_n x^{-n}$ is said to be an asymptotic expansion of a function $f(x)$ as $x\to \infty$ if for each $N\in\bN$
		\bes
			f(x)=\sum_{n=0}^{N}c_nx^{-n} +o(x^{-N}).
		\ees
\end{definition}
Main properties of asymptotic series and asymptotic expansions can be found in \cite{Erd}.
Taylor series expansion can also be seen as an asymptotic expansion but the converse is not generally true and the asymptotic series may also be divergent. The main characteristic of asymptotic expansion is that it provides good approximation using finite number of terms while letting $x\to\infty$.

Beacause of the intrinsity \eqref{defM-int}, 
mean $M$ would possess the asymptotic power series as $x\to \infty$ of the form
\bes
	M(x+s,x+t) = \sum_{n=0}^\infty c_n(s,t) x^{-n+1}
\ees
with $c_0(s,t)=1$. For a homogeneous symmetric mean the coefficients $c_n(s,t)$ are also homogeneous symmetric polynomials of degree $n$ in variables $s$ and $t$ and for $s=-t$ hey have simpler form. Let the means included possess the asymptotic expansions as $x\to \infty$ of the form
\begin{align}
	M_0(x-t,x+t) &= \sum_{n=0}^\infty c_n t^{2n}x^{-2n+1} \label{asym-exp-M0},\\
	M_1(x-t,x+t) &= \sum_{n=0}^\infty a_n t^{2n}x^{-2n+1} ,\notag \\ 
	M_2(x-t,x+t) &= \sum_{n=0}^\infty b_n t^{2n}x^{-2n+1} .\notag 
\end{align}
Conversely, it can also be shown that the expansion in variables $(x-t,x+t)$ is sufficent to obtain so called two variable expansion, i.e.\ the expansion in variables $(x+s,x+t)$.
Furthermore, note that
\be\label{a0b0c0=1}
	a_0=b_0=c_0=1.
\ee

In this section we will find the asymptotic expansions of means $M_2^S=S_{M_0}(M_1)$ and $M_2^\sigma=\sigma_{M_0}(M_1)$.

\subsection{Symmetry $S_{M_0}$}
 Recall the recently developed results for tansformations of asymptotic series, i.e.\ the complete asymptotic expansions of the quotient and the power of asymptotic series.

\begin{lemma}[\cite{ElVu-04}, Lemma 1.1.]\label{lemma-quot}
	Let function $f(x)$ and $g(x)$ have following asymptotic expansions ($a_0\neq0, b_0\neq0$) as $x\to\infty$:
	\bes
		f(x)\sim\sum_{n=0}^\infty a_n x^{-n},\qquad g(x)\sim\sum_{n=0}^\infty b_n x^{-n}.
	\ees
	Then asymptotic expansion of their quotient $f(x)/g(x)$ reads as
	\bes
		 \frac{f(x)}{g(x)}\sim \sum_{n=0}^\infty c_n x^{-n},
	\ees
	where coefficients $c_n$ are defined by
	\bes
		c_n=\frac{1}{b_0}\left(a_n-\sum_{k=0}^{n-1} b_{n-k} c_{k}\right).
	\ees
\end{lemma}

\begin{lemma}[\cite{ChenElVu-2013,Gould-1974}]
    \label{lemma-power}
    Let $m(x)$ be a function with asymptotic expansion $(c_0\ne0)$:
    $$
        m(x)\sim   \sum_{n=0}^\infty c_n x^{-n}, \qquad (x\to\infty).
    $$
    Then for all real $r$ it holds
    $$
        [m(x)]^{r}\sim \sum_{n=0}^\infty P[n,r,(c_j)_{j\in\bN_0}] x^{-n}
    $$
    where
    \begin{equation}
    \begin{aligned}
        P[0,r,(c_j)_{j\in\bN_0}]&=c_0^{r},\\
        P[n,r,(c_j)_{j\in\bN_0}]&=\frac1{n c_0}\sum_{k=1}^n[k(1+r)-n]c_k P[n-k,r,(c_j)_{j\in\bN_0}].
    \end{aligned}
    \label{lemma-power-coeffP}
    \end{equation}
\end{lemma}

Symmetric mean with respect to mean $M_0$ of mean $M_1$ via the group structure $(\cM_\cD, \ast)$
as a consequence of \eqref{SM0-explicit} can be expressed as
\bes\begin{aligned}
	M_2^S&(x-t,x+t) = S_{M_0} (M_1) (x-t,x+t) \\
		&= \frac{(x-t)(M_1-x+t)(M_0-x-t)^2-(x+t)(M_0-x+t)^2(M_1-x-t)}%
		{(M_1-x+t)(M_0-x-t)^2-(M_0-x+t)^2(M_1-x-t)}\\
		&= \frac{(x-t)(\overline M_1+t)(\overline M_0-t)^2-(x+t)(\overline M_0+t)^2(\overline M_1-t)}%
		{(\overline M_1+t)(\overline M_0-t)^2-(\overline M_0+t)^2(\overline M_1-t)}\\
		&= x+ 
		\frac{2t^2\overline M_0 - t^2\overline M_1- \overline M_0^2\overline M_1}{t^2+\overline M_0^2-2\overline M_0 \overline M_1},
\end{aligned}
\ees
where $\overline M_i$, $i=1,2,3$, stands for $M_i-x$. 
The variables $(x-t,x+t)$ were omitted for the sake of symplicity.
Further calculations reveal that
\bes
\begin{aligned}
	&M_2^S (x-t,x+t) 
		= x+t^2x^{-1}  
		  \Bigl[ (2c_1-a_1)+\\ 
		  &  +\sum_{n=0}^\infty\Bigl(2c_{n+2}-a_{n+2}
			+\sum_{k=0}^n\Bigl(\sum_{j=0}^k \left(c_{j+1}c_{k-j+1}\right)a_{n+1-k}\Bigr)\Bigr) 
		   t^{2n+2}x^{-2n-2} \Bigr] \times\\ 
		 &  \times \Bigl[{1+\sum_{n=0}^\infty \sum_{k=0}^n c_{k+1}(c_{n-k+1}-2a_{n-k+1})t^{2n+2}x^{-2n-2}}\Bigr]^{-1}. 
\end{aligned}
\ees
Coefficients $b_n^S$ for $n\ge1$ are obtained 
using Lemma \ref{lemma-quot} for division of asymptotic series.
Hence, we have the following:
\begin{align*}
	b_{0}^S&=1,\\
	b_{n}^S&= num_n-\sum_{k=0}^{n-2} den_{n-1-k} b_{k+1}^S,\quad n\ge1,
\end{align*}
where $(num_{n})_{n\in\bN_0}$ and $(den_{n})_{n\in\bN_0}$ dentote auxiliary sequences wihich appear in numerator and denominator:
\bes
\begin{aligned}
	num_0&=2c_1-a_1,\\
	num_{n}&=2c_{n+1}-a_{n+1}
			+\sum_{k=0}^{n-1}\left(\sum_{j=0}^k (c_{j+1}c_{k-j+1})a_{n-k}\right),\quad n\ge1,
\end{aligned}
\ees
and
\bes
\begin{aligned}
	den_0&=1,\\
	den_{n}&=\sum_{k=0}^{n-1} c_{k+1}(c_{n-k}-2a_{n-k}),\quad n\ge1.
\end{aligned}
\ees
We shall calculate the first few coefficients:
\bes
\begin{aligned}
 		b_0^S &= 1, \\
		b_1^S &= 2c_1-a_1,\\
		b_2^S &= 2c_2-a_2 -2c_1(a_1-c_1)^2,\\
 		b_3^S &= 2c_3-a_3-2(a_1-c_1)( 2a_2c_1+c_1^2(2a_1^2-3a_1c_1+c_1^2)+(a_1-3c_1)c_2),\\
 		b_4^S &= 2c_4-a_4-2(a_2^2 c_1+4 a_1^4 c_1^3+4 a_1^3 c_1 (-3 c_1^3+c_2)\\
 			&\quad +2 a_2 ((3 a_1-2 c_1) (a_1-c_1) c_1^2
 			 +(a_1-2 c_1) c_2)\\
 			 &\quad +a_1^2 (13 c_1^5-15 c_1^2 c_2+c_3) 
 			  +2 a_1 (a_3 c_1-3 c_1^6+8 c_1^3 c_2-c_2^2-2 c_1 c_3)\\
 			&\quad +c_1 (-2 a_3 c_1+c_1^6-5 c_1^3 c_2+3 c_2^2+3 c_1 c_3)), \\
 		b_5^S &= 2c_5 -a_5 -2(-2 a_4 c_1^2+8 a_1^5 c_1^4+4 a_3 c_1^4-c_1^9-4 a_3 c_1 c_2+7 c_1^6 c_2-10 c_1^3 c_2^2\\
 			&\quad +c_2^3 +a_2^2 (6 a_1 c_1^2-5 c_1^3+c_2)+4 a_1^4 c_1^2 (-7 c_1^3+3 c_2)-5 c_1^4 c_3+6 c_1 c_2 c_3\\
 			&\quad +2 a_1^3 (19 c_1^6-24 c_1^3 c_2+c_2^2+2 c_1 c_3)+2 a_2 (a_3 c_1+8 a_1^3 c_1^3-3 c_1^6+8 c_1^3 c_2\\
 			&\quad -c_2^2 +6 a_1^2 c_1 (-3 c_1^3+c_2)-2 c_1 c_3+a_1 (13 c_1^5-15 c_1^2 c_2+c_3))+3 c_1^2 c_4\\
 			&\quad +a_1^2 (6 a_3 c_1^2-5 c_1 (5 c_1^6-13 c_1^3 c_2+3 c_2^2+3 c_1 c_3)+c_4)+2 a_1 (a_4 c_1\\
 			&\quad+a_3 (-5 c_1^3+c_2) 
 			 +2 (2 c_1^8-9 c_1^5 c_2+6 c_1^2 c_2^2+4 c_1^3 c_3-c_2 c_3-c_1 c_4))).
 \end{aligned}
 \ees

\subsection{Symmetry $\sigma_{M_0}$} 
The problem of functional symmetic mean corresponds the functional equation
\bes
	M_0(x-t,x+t) = M_0(M_1(x-t,x+t),M_2(x-t,x+t))
\ees
which we will solve in terms of asymptotic series.
To this end, we shall use the following result from Burić and Elezović about the asymptotic expansion of the composition of means.
\begin{theorem}[\cite{BurEl-2017}, Theorem 2.2.]\label{thm-BurEl-compound}
	Let $M$ and $N$ be given homogeneous symmetric means with asymptotic expansions
	\bes 
		M(x-t,x+t) = \sum_{k=0}^\infty a_k t^{2k}x^{-2k+1},\ \ 
		N(x-t,x+t) = \sum_{k=0}^\infty b_k t^{2k}x^{-2k+1},
	\ees
	and let $F$ be homogeneous symmetric mean with expansion
	\bes 
		F(x-t,x+t) = \sum_{k=0}^\infty \gamma_k t^{2k}x^{-2k+1}.
	\ees
	Then the composition $H=F(M,N)$ has asymptotic expansion
	\bes 
		H(x-t,x+t) = \sum_{k=0}^\infty h_n t^{2n}x^{-2n+1},
	\ees
	where coefficients $(h_n)$ are calculated by
	\bes
		h_n=\sum_{k=0}^{\lfloor{\frac{n}{2z}} \rfloor} \gamma_k\sum_{j=0}^{n-2zk} 
		P[j,2k,(d_m)_{m\in\bN_0}]P[n-2zk-j,-2k+1,(c_m)_{m\in\bN_0}].
	\ees
	Sequences $(c_n)$ and $(d_n)$ are defined by
	\bes
		c_n=\frac12(a_n+b_n),\ \ d_n=\frac12(a_{n+z}-b_{n+z}),\quad n\ge0,
	\ees
	where $z$ is the smallest number such that $d_n\neq0$.
\end{theorem}

Applying Theorem \ref{thm-BurEl-compound} on $M=M_1$, $N=M_2$ (or equivalently $M=M_2$, $N=M_1$) and $F=M_0$ we obtain the asymptotic expansion of the composition $M_0(M_1,M_2)$. 
Since the equation $M_0=M_0(M_1,M_2)$ holds, on the other side in Theorem \ref{thm-BurEl-compound} we also have $H=M_0$. The coeficients in the asymptotic expansion of the composition $M_0(M_1,M_2)$ equal the coefficients $c_n$ in the asymptotic expansion of mean $M_0$. In the end, we obtain the recursive algorithm for coefficients $c_n$:
\be\label{formula-sigma-cn}
	\begin{aligned}
		c_0&=1;\\
		c_n&= \sum_{k=0}^{\lfloor{\frac{n}{2z}} \rfloor} c_k\sum_{j=0}^{n-2zk} 
		P[j,2k,(\tfrac12(a_m-b_m^\sigma))_{m\ge z}] \times\\
			&\qquad \times P[n-2zk-j,-2k+1,(\tfrac12(a_m+b_m^\sigma))_{m\in\bN_0}],\quad n\ge1,
	\end{aligned}
\ee
where $P[n,r,(c_m)_{m\in\bN_0}]$, $n\in\bN_0$, denotes the $n$-th coefficient in the asymptotic expansion of $r$-th power of the asymptotic seires with coefficients $(c_m)_{m\in\bN_0}$, as it was defined in \eqref{lemma-power-coeffP}. Because of \eqref{a0b0c0=1}, $z$ is always greater or equal to 1.

For $z=1$ we calculate the first few coefficients:
\bes
 \begin{aligned}
 		c_0 &= 1, \\
 		c_1 &= \frac12(a_1+b_1^\sigma), \\
 		c_2 &= \frac12 (a_2+b^\sigma_2)+\frac14 (a_1-b^\sigma_1)^2 c_1,  \\
 		c_3 &= \frac12 (a_3+b^\sigma_3)- \frac18 (a_1-b^\sigma_1) (a_1^2-4 a_2-(b^\sigma_1)^2+4 b^\sigma_2) c_1,\\
 		c_4 &= \frac12(a_4+b^\sigma_4)+\frac1{16}((a_1^4+4 a_2^2-8 a_3 b^\sigma_1+(b^\sigma_1)^4+2 a_2 ((b^\sigma_1)^2-4 b^\sigma_2)\\
 		&\quad -2 a_1^2 (3 a_2+(b^\sigma_1)^2-b^\sigma_2)-6 (b^\sigma_1)^2 b^\sigma_2+4 (b^\sigma_2)^2\\
 		&\quad+4 a_1 (2 a_3+b^\sigma_1 (a_2+b^\sigma_2)-2 b^\sigma_3)
 		 +8 b^\sigma_1 b^\sigma_3) c_1+(a_1-b^\sigma_1)^4 c_2),\\
 		c_5 &= \frac12(a_5+b^\sigma_5)-\frac1{32}((a_1^5+a_1^4 b^\sigma_1-4 a_2^2 b^\sigma_1+16 a_4 b^\sigma_1-4 a_3 (b^\sigma_1)^2+(b^\sigma_1)^5\\
 		&\quad -2 a_1^3 (4 a_2+(b^\sigma_1)^2)+16 a_3 b^\sigma_2-8 (b^\sigma_1)^3 b^\sigma_2+12 b^\sigma_1 (b^\sigma_2)^2\\
 		&\quad -8 a_2 (2 a_3+b^\sigma_1 b^\sigma_2-2 b^\sigma_3)
 		+2 a_1^2 (6 a_3-(b^\sigma_1)^3+4 b^\sigma_1 b^\sigma_2-2 b^\sigma_3)+12 (b^\sigma_1)^2 b^\sigma_3\\
 		&\quad -16 b^\sigma_2 b^\sigma_3-16 b^\sigma_1 b^\sigma_4
 		+a_1 (12 a_2^2-16 a_4-8 a_3 b^\sigma_1+(b^\sigma_1)^4 
 		+8 a_2 ((b^\sigma_1)^2-b^\sigma_2)\\
 		&\quad -4 (b^\sigma_2)^2-8 b^\sigma_1 b^\sigma_3+16 b^\sigma_4)) c_1  
 		 -(a_1-b^\sigma_1)^3 (3 a_1^2-8 a_2-3 (b^\sigma_1)^2+8 b^\sigma_2) c_2).
 \end{aligned}
 \ees

The connetcion between $b_n^\sigma$ and $c_n$ with the highest index $n$ in each equation is linear. In the expression \eqref{formula-sigma-cn} $b_n^\sigma$ appears ony in the second part  
	\be\label{form35}
		P[n-2zk-j,-2k+1,(\tfrac12(a_m+b_m^\sigma))_{m\in\bN_0}],
	\ee
 when $k=j=0$. Then \eqref{form35} becomes $P[n,1,(\tfrac12(a_m+b_m^\sigma))_{m\in\bN_0}]$ which represents the $n$-th coefficient in the 
$\sum_{n=0}^\infty \frac12(a_n+b_n^\sigma)t^{2n}x^{-2n+1}$
to the power of 1 which equals $\tfrac12(a_n+b_n^\sigma)$.
So we can easily extract $b_n^\sigma$. The first few coefficients $b_n^\sigma$ are:
 \begin{align*} 
 		b_0^\sigma &= 1, \\
		b_1^\sigma &= 2c_1-a_1,\\
		b_2^\sigma &= 2c_2-a_2 -\frac12c_1(a_1-b_1^\sigma),\\
 		b_3^\sigma &= 2c_3-a_3+\frac14(a_1-b_1^\sigma)(a_1^2-4a_2-({b_1^\sigma})^2+4b_2^\sigma)c_1,\\
 		b_4^\sigma &= 2c_4-a_4-\frac1{8}((a_1^4+4 a_2^2-8 a_3 b^\sigma_1+(b^\sigma_1)^4+2 a_2 ((b^\sigma_1)^2-4 b^\sigma_2)\\
 		&\quad -2 a_1^2 (3 a_2+(b^\sigma_1)^2-b^\sigma_2)
 		 -6 (b^\sigma_1)^2 b^\sigma_2+4 (b^\sigma_2)^2\\
 		 &\quad +4 a_1 (2 a_3+b^\sigma_1 (a_2+b^\sigma_2)-2 b^\sigma_3)
 		 +8 b^\sigma_1 b^\sigma_3) c_1+(a_1-b^\sigma_1)^4 c_2),\\
		b_5^\sigma &= 2c_5-a_5+\frac1{16}((a_1^5+a_1^4 b^\sigma_1-4 a_2^2 b^\sigma_1+16 a_4 b^\sigma_1-4 a_3 (b^\sigma_1)^2+(b^\sigma_1)^5\\
 		&\quad -2 a_1^3 (4 a_2+(b^\sigma_1)^2)+16 a_3 b^\sigma_2-8 (b^\sigma_1)^3 b^\sigma_2+12 b^\sigma_1 (b^\sigma_2)^2\\
 		&\quad -8 a_2 (2 a_3+b^\sigma_1 b^\sigma_2-2 b^\sigma_3)
 		+2 a_1^2 (6 a_3-(b^\sigma_1)^3+4 b^\sigma_1 b^\sigma_2-2 b^\sigma_3)+12 (b^\sigma_1)^2 b^\sigma_3\\
 		&\quad -16 b^\sigma_2 b^\sigma_3-16 b^\sigma_1 b^\sigma_4
 		 +a_1 (12 a_2^2-16 a_4-8 a_3 b^\sigma_1+(b^\sigma_1)^4
 		+8 a_2 ((b^\sigma_1)^2-b^\sigma_2)\\
 		&\quad -4 (b^\sigma_2)^2
 		 -8 b^\sigma_1 b^\sigma_3+16 b^\sigma_4)) c_1
 		-(a_1-b^\sigma_1)^3 (3 a_1^2-8 a_2-3 (b^\sigma_1)^2+8 b^\sigma_2) c_2).
 \end{align*}

For beter understanding the role of the parameter $z$ we shall recall the idea behind the proof of Theorem \ref{thm-BurEl-compound}. The composition $F(M,N)$ has the asymptotic expansion
\bes
\begin{aligned}
	F(M(x-t,x+t)&,N(x-t,x+t)) =\\
		&=
		  F\left(\frac{M+N}2-\frac{N-M}2,\frac{M+N}2+\frac{N-M}2 \right)\\
		&= \sum_{k=0}^\infty \gamma_k 
			\left(\frac{N-M}2\right)^{2k}
			\left(\frac{M+N}2\right)^{-2k+1}.
\end{aligned}
\ees 
Larger $z$ corresponds with the equating $a_i$ and $b_i^\sigma$ and some parts of the coefficients $c_n$ reduce. Observation of the cases with $z>1$ in sequel did not provide any new information about the ceofficients $c_n$.

\subsection{Comparison of coefficients} 
Sequences $(b_n^S)_{n\in\bN_0}$ and $(b_n^\sigma)_{n\in\bN_0}$ represent the coefficients in asymptotic expansions of means which are results of mappings $S_{M_0}(M_1)$ and $\sigma_{M_0}(M_1)$ respectively.
Since we are looking for a mean $M_0$ such those mappings coincide, $b_n^S$ and $b_n^\sigma$ need to be equal. Since the equality must hold for any mean $M_1$ we may suppose that $z=1$ which is equivalent with $a_1\neq c_1$.
Equating $b_0^S$ with $b_0^\sigma$ and $b_1^S$ with $b_1^\sigma$ doesn't provide any new information except
\bes
	b_0=b_0^S=b_0^\sigma= 1  \text{ and } b_1=b_1^S=b_1^\sigma= 2c_1-a_1.
\ees
With such $b_1^\sigma$ we may express $b_2^\sigma$ as
\bes
		b_2^\sigma = 2c_2-a_2 -2c_1(a_1-c_1)^2,
\ees
which is already equal to $b_2^S$.
Now we can substitute
\bes
	b_2=b_2^S=b_2^\sigma= 2c_2-a_2 -2c_1(a_1-c_1)^2,
\ees
in $b_3^\sigma$ to obtain
\bes
	b_3^\sigma= 2c_3-a_3-2c_1(a_1-c_1)(2a_2+2c_1(a_1-c_1)^2+c_1^2-a_1c_1 -2c_2),
\ees
which after equating with $b_3^S$ gives the following condition
\bes
 	(a_1-c_1)^2(c_1^2+c_1^3+c_2)=0.
\ees
Since we assumed that $a_1$ and $c_1$ are not equal, it is necessarily
\bes
	c_2=-c_1^2(1+c_1).
\ees
Now we have
\bes
	b_3=b_3^S=b_3^\sigma=2 c_3-a_3 -2 c_1 (a_1-c_1) \left((3-4 a_1) c_1^2+a_1 (2 a_1-1) c_1+2 a_2+4 c_1^3\right).
\ees
After substitutions we observe the next coefficient
\bes
	\begin{aligned}
	b_4^\sigma &= 2 c_4-a_4 
		-2 c_1 \bigl(2 a_2 c_1 \bigl(-a_1 (6 c_1+1)+3 a_1^2+2 c_1 (2 c_1+1)\bigr)\\
		& +c_1 \bigl(c_1 \bigl(-4 a_1^3 (4 c_1+1)+a_1^2 (2 c_1+1) (15 c_1+2)-2 a_1 c_1 (14 c_1 (c_1+1)+3) \\
		& +4 a_1^4+c_1^2 (c_1 (11 c_1+15)+5)\bigr) -2 a_3\bigr)+2 c_3 (c_1-a_1)+a_2^2+2 a_1 a_3\bigr)
	\end{aligned}
\ees
which after equating with $b_4^S$ gives the following condition:
\bes
 	(a_1-c_1)^2 \left(2 c_1^3 (c_1+1)^2-c_3\right) =0,
\ees
and we conclude that it must be
\bes
	c_3=2c_1^3(1+c_1)^2.
\ees
We continue with this procedure as it was described above.
Further calculations reveal that the first few coefficients $c_n$ have the following form:
\begin{align*}
	c_0 &=1,\\
	c_1 &= c,\\
	c_2 &= -c^2(1+c),\\
	c_3 &= 2c^3(1+c)^2,\\
	c_4 &= -5c^4(1+c)^3,\\
	c_5 &= 14c^5(1+c)^4,\\
	c_6 &= -42c^6(1+c)^5.
\end{align*}

After these first steps it is natural to state the following hypothesis about the general formula for the coefficients in the asymptotic expansion of mean $M_0$:
\begin{align}\label{hypo-cn}
	c_0 &=1, \notag\\
	c_n &= (-1)^{n-1} C_{n-1} c^n (1+c)^{n-1}, \ n\ge1,
\end{align}
where $C_n$ denotes the $n$-th Catalan number. Catalan numbers appear in many occasions and their behavior has been widely explored. Here we mention only a few properties which we will use in sequel.
Catalan numbers are defined by
\bes
	C_n=\frac1{n+1}\binom{2n}{n},\quad n\in\bN_0
\ees
and they satisfy the recursive relation
\bes
	C_{n+1} = \sum_{k=0}^n C_k C_{n-k},\quad n\in\bN_0.
\ees
Based on this recursive relation the generating function for Catalan numbers can be obtained (\cite{GKnuthP}):
\be\label{catalan-gf}
	\sum_{n=0}^\infty C_n y^n = \frac{1-\sqrt{1-4y}}{2y}.
\ee

\section{New mean function}
In this section we shall find closed a form for a mean whose coefficients are given in  \eqref{hypo-cn}.
We start from asymptotic expansion \eqref{asym-exp-M0}:
	\begin{align}
		M_0(x-t,x+t) &= x+ \sum_{n=1}^\infty(-1)^{n-1}C_{n-1}c^{n}(1+c)^{n-1}t^{2n}x^{-2n+1} 
			\label{asym-exp-M0-Cn}\notag\\
			&=  x+ \sum_{n=0}^\infty(-1)^{n}C_{n}c^{n+1}(1+c)^{n}t^{2n+2}x^{-2n-1}  \notag\\
			&= x+ct^2x^{-1} \sum_{n=0}^\infty C_{n} \left[ - \frac{c(1+c)t^2}{x^2}\right]^{n}.
	\end{align}
Introducing the substitution $y= - \frac{c(1+c)t^2}{x^2}$ yields
	\begin{align*}
		M_0(x-t,x+t) 
			&= x+ct^2x^{-1} \sum_{n=0}^\infty C_{n}y^{n},
	\end{align*}
and then according to the formula \eqref{catalan-gf}, for $c+1\neq0$, we obtain
	\begin{align}\label{theMxt}
		M_0(x-t,x+t) 
			&= x+ct^2x^{-1} \frac{1-\sqrt{1-4y}}{2y}\notag\\
			&= \frac{1+2c}{2(1+c)} x + \frac{1}{2(1+c)}\sqrt{x^2+4c(1+c)t^2}.
	\end{align}
With substitution 
\bes
	x=\frac{a+b}2,\quad t=\frac{b-a}2,
\ees
in \eqref{theMxt}, we obtain the expression for $M_0$ in terms of variables $a$ and $b$.
	For $c\in\bR\setminus\{-1\}$ and $a,b>0$ we define function
	 $L_c \colon \bR^+\times \bR^+ \to \bR^+$
	\be\label{theM}
		L_c(a,b) = \frac{a+b}2 \frac{1+2c}{2(1+c)} + \frac1{2(1+c)} \sqrt{\left(\frac{a+b}2\right)^2+4c(1+c)\left(\frac{b-a}2\right)^2}.
	\ee
\begin{theorem}{}
	For $c\in \langle-1,+\infty\rangle$ funcion $L_c$ is a mean.
\end{theorem}	

\begin{remark}
	For $c=-1$ function $L_c$ corresponds to the harmonic mean.
\end{remark}

\begin{proof}
We shall divide proof into the several parts.

{\bf 1.\ Function $L_c$ is well defined for all $(a,b) \in \bR^+\times \bR^+ $.}
We can rearrange terms under the square root
\begin{align*}
	\left(\frac{a+b}2\right)^2+4c(1+c)\left(\frac{b-a}2\right)^2 
		&=\frac14 \left( ({a+b})^2+4c(1+c)({b-a})^2  \right) \notag\\
		&=\frac14 (1+2c)^2(a-b)^2+4ab >0.
\end{align*}

{\bf 2.\ Function $L_c$ is a mean.}
$L_c$ is obvoiusly a symmetric function.
Without the loss of generality we may suppose that $a < b$. Now the condition 
\bes
	\min(a,b)< L_c(a,b)<\max(a,b)
\ees
	is equivalent to
\bes
	a < \frac{a+b}2 \frac{1+2c}{2(1+c)} + \frac1{2(1+c)} \sqrt{\left(\frac{a+b}2\right)^2+4c(1+c)\left(\frac{b-a}2\right)^2}  < b.
\ees
Let $s=\frac a b$. Then $0<s<1$ and previous expression becomes
\bes
	s < \frac{s+1}2 \frac{1+2c}{2(1+c)} + \frac1{4(1+c)} \sqrt{(1+s)^2+4c(1+c)(1-s)^2} < 1,
\ees
or in other words
\begin{align*}
	\frac1{4(1+c)} [(2c+3)s-(2c+1)] &< \frac1{4(1+c)} \sqrt{(1+s)^2+4c(1+c)(1-s)^2} \notag\\
		&< \frac1{4(1+c)}[-(1+2c)s+(2c+3)].
\end{align*}
Denote
\bes 
	I_1^c(s)  = (2c+3)s-(2c+1),\ \ 
	I_2^c(s) = -(1+2c)s+(2c+3).
\ees

Suppose $c+1>0$. We need to prove the following inequalities
\be\label{ineq-i-s-c>0}
		I_1^c(s) <  \sqrt{(1+s)^2+4c(1+c)(1-s)^2}< I_2^c(s).
\ee
If $I_1^c(s) \ge0$, squaring the left side inequality yields
\bes
	(2c+3)^2s^2-2s(2c+3)(2c+1)+(2c+1)^2 < (1+s)^2+4c(1+c)(1-s)^2,
\ees
which reduces to the condition
\bes
	8s(s-1)(c+1) <0,
\ees
which is fulfilled.
If $I_1^c(s) < 0$ then the left side inequality in \eqref{ineq-i-s-c>0} is trivially satisfied. 

On the other side, 
\bes
	I_2^c(s) \ge 0 \Leftrightarrow
		 \left(c<-\frac12 \text{ and } s\ge \frac{2c+3}{2c+1}\right)
		  \text{ or } \left(c> -\frac12 \text{ and } s\le \frac{2c+3}{2c+1}\right) 
		  \text{ or } c= -\frac12.
\ees
Since for $c\in\langle -1,-\frac12\rangle$ we have $\frac{2c+3}{2c+1}<0$
and for $c\in\langle-\frac12,+\infty\rangle$ $\frac{2c+3}{2c+1}>1$,
it holds that $I_2^c(s) \ge 0$.
Squaring the right side inequality in \eqref{ineq-i-s-c>0} then yields 
\bes
	(1+s)^2+4c(1+c)(1-s)^2 < (1+2c)^2s^2-2s(1+2c)(2c+3)+(2c+3)^2
\ees
which reduces to obviously fulfilled condition
\bes
	8(s-1)(c+1)<0.
\ees
	Proof of the theorem is complete.
\end{proof}

\begin{remark}
Notice that we proved that $L_c$ is a strict mean, i.e.\ for $s\neq t$ strict inequalities hold:
\bes
	\min(s,t) < M(s,t) < \max(s,t).
\ees
\end{remark}

\subsection{Special cases}
Before we continue further, let us see what happens with some of the special cases of parameter $c$.
\begin{example} 
	\begin{enumerate} 
		\item $c=-1$. Then mean has two non-zero coefficients:
			\bes
				c_0=1,\quad c_1=c,\quad c_n=0,\ n\ge2.
			\ees
			Corresponding asymptotic expansion is finite. From \eqref{asym-exp-M0-Cn} we obtain
			\bes
				L_c(x-t,x+t) =x+ ct^2x^{-1},
			\ees
			which, after substitution $x=\frac{a+b}2,\ t=\frac{b-a}2$, becomes
			\bes
				L_c(a,b) =  \frac{a+b}2+c\cdot \frac{(b-a)^2}4\cdot\frac2{a+b}
					=\frac{2ab}{a+b}=H(a,b),
			\ees
			where $H$ is the harmonic mean.
		\item $c=0$. All coefficients except $c_0$ equal zero. Then either from the \eqref{asym-exp-M0-Cn} or \eqref{theMxt} we obtain
			\bes
				L_c(x-t,x+t) =x,
			\ees
			and after the substitution
			\bes
				L_c(a,b) =  \frac{a+b}2 = A(a,b),
			\ees
			where $A$ is the arithmetic mean.
		\item $c=-\frac12$. The coefficients are
			\be\label{cn-geom}
				c_0=1,\quad c_n=-\frac1{2^{2n-1}} C_{n-1}, \ n\ge1.
			\ee
			Coefficients \eqref{cn-geom} correspond to the coefficients in asymptotic expansion of geometric mean obtained in \cite{ElVu-09} for $\alpha=0$ and $\beta=t$, and also to coefficients of power mean $M_p$ with $p=0$ obtained in \cite{ElVu-04}.
			On the other side, from the formula \eqref{theMxt} we obtain
			\bes
				L_c(x-t,x+t) =\sqrt{x^2-t^2},
			\ees
			and after substitution
			\bes
				L_c(a,b) = \sqrt{ab} = G(a,b),
			\ees
			where $G$ is the geometric mean.
	\end{enumerate}
\end{example}

From the example above we see that we covered the cases of means for which in \cite{Farhi} was stated that symmetries $S$ and $\sigma$ coincide.

\subsection{Limit cases and monotonicity}
In this subsection we study properties of $L_c$ with respect to parameter $c$.
First, we state the following proposition which can be proved using basic methods of mathematical analysis.
\begin{proposition}\label{prop-limit}
	For a fixed pair $(a,b)\in \bR^+\times \bR^+$ and function $L_c$ holds
	\begin{enumerate}
		\item $\displaystyle\lim_{c\to-\infty} L_c(a,b) = \lim_{c\to+\infty}  L_c(a,b)= \max(a,b)$,
		\item $\displaystyle\lim_{c\to-1-} L_c(a,b) = +\infty$,
		\item $\displaystyle\lim_{c\to-1+} L_c(a,b) = \frac{2ab}{a+b} = H(a,b)$.
	\end{enumerate}
\end{proposition}

It is well known that the following double inequality hold
\bes
	H<A<G.
\ees
Also, $H=L_c$ for $c\to-1$, $G=L_c$ for $c=-\frac12$ and $A=L_c$ for $c=0$.
In the next Theorem we explore the ordering of means $L_c$ with respect to parameter $c$.

\begin{theorem}
	For a fixed pair $(a,b)\in \bR^+\times \bR^+$, $a\neq b$, function $f\colon \bR\setminus\{-1\} \to\bR$,
	\bes
		f(c)=L_c(a,b) 	
	\ees
	is strictly increasing on intervals $\langle -\infty,-1\rangle$ and $\langle-1,+\infty\rangle$.
	In addition, for $c_1\in \langle -\infty,-1\rangle$ and $c_2\in \langle-1,+\infty\rangle$ it holds	$f(c_1)\ge f(c_2)$.
\end{theorem}

\begin{proof}
	From the definition of $L_c$ it follows
	\bes
		f(c)= (a+b)g(c)+\sqrt{(a+b)^2 g(c)^2-4ab\, g(c) +ab},
	\ees
	where
	\bes
	 	g(c)=\frac{1+2c}{4(1+c)}.
	\ees
	The first derivative od function $f$ equals
	\bes
		f'(c)= g'(c)\left[ (a+b)+\frac{(a+b)^2g(c)-2ab}{\sqrt{(a+b)^2 g(c)^2-4ab\, g(c) +ab}} \right].
	\ees
	We easily see that
	\bes
		g'(c)=\frac14 \left(\frac{1+2c}{1+c}\right)' = \frac14 \cdot\frac1{(1+c)^2} >0.
	\ees
	So the condition  $f'(c)>0$ is equivalent to
	\bes
		g(c) > (a+b)+\frac{(a+b)^2g(c)-2ab}{\sqrt{(a+b)^2 g(c)^2-4ab\, g(c) +ab}} >0
	\ees
	or 
	\be\label{mono-inequal}
		(a+b)\sqrt{(a+b)^2 g(c)^2-4ab\, g(c) +ab} >  2ab-(a+b)^2 g(c) ,
	\ee
	which is satisfied when the right side is less than zero.
	On the other side, when 
	$$2ab-(a+b)^2 g(c) \ge 0,$$
	condition \eqref {mono-inequal} is equivalent to
	\begin{align*}
		(a+b)^2 & \left[(a+b)^2 g(c)^2 -4ab\, g(c) +ab\right] \\
		& > 4a^2b^2 +(a+b)^2 \left[(a+b)^2g(c)^2-4ab\, g(c)\right]
	\end{align*}
which is true for $a\neq b$. 

The second part of the theorem follows from the first part and the Proposition \ref{prop-limit}.
\end{proof}


\section{Answer to the open question}
\begin{theorem}{}
	For mean $L_c$, $c\in [-1,+\infty\rangle$, defined in \eqref{theM}
	symmetries $S_{L_c}$ i $\sigma_{L_c}$ coincide.
\end{theorem}

\begin{proof}
	Let us rewrite mean $L_c$ in the following manner:
	\bes
		L_c(a,b) = \frac1{4(1+c)}\left[ (1+2c)(a+b)+\sqrt{(a+b)^2+4c(1+c)(b-a)^2} \right].
	\ees
	For $M_0=L_c$ and variable mean $M_1=M$, there exists symmetric mean $\sigma=\sigma_{L_c}(M)$, i.e. the condition $L_c(M,\sigma)=L_c$ holds, which yields (for the sake of brevity, the variables will be ommited):
	\bes
		\frac1{4(1+c)}\left[ (1+2c)(M+\sigma)+\sqrt{(M+\sigma)^2+4c(1+c)(M-\sigma)^2} \right]
			=L_c,
	\ees
	or equivalently
	\bes
		\sqrt{(M+\sigma)^2+4c(1+c)(M-\sigma)^2}	=4(1+c) L_c -(1+2c)(M+\sigma).
	\ees
	We rearrange the terms and because of the existence of mean $\sigma=\sigma_{L_c}(M)$, we may square the latter expression:
	\begin{align*}  
		M^2&(1+2c)^2+2M\sigma(1-4c-4c^2)+\sigma^2(1+2c)^2\\\notag
			&= \left[ 4(1+c) L_c -(1+2c)M\right]^2 -2  \left[ 4(1+c) L_c -(1+2c)M\right]^2 
				+\sigma^2(1-2c)^2.
	\end{align*}
	The terms $\sigma^2(1-2c)^2$ cancel from both sides.
	Further calculation gives
	\begin{align*}
		2M(1-4c-4c^2)&\sigma+2\big(4(1+c)L_c-(1+2c)M\big)(1+2c)\sigma \notag\\
			&= -M^2(1+2c)^2+\big(4(1+c)L_c-(1+2c)M\big)^2,
	\end{align*}
	and finally
	\be\label{sigma-explicit}
	  \sigma= \frac{L_c \big((1+2c)M- 2(1+c)L_c\big)}{2cM-(1+2c)L_c}.
	\ee
	Thus we obtained the explicit expression for mean $\sigma=\sigma_{L_c}(M)$ in terms of $M$ and $L_c$.
	
	On the other side, from \eqref{SM0-explicit} we know that
	\bes
		S_{L_c}(M)= \frac{a(M-a)(L_c-b)^2-b(L_c-a)^2(M-b)}{(M-a)(L_c-b)^2-(L_c-a)^2(M-b)},
	\ees
	which may be written as
	\be\label{SMc-explicit-k}
		S_{L_c}(M) = \frac{K_1M-K_2}{K_0 M -K_1},
	\ee
	where
	\begin{align*}
		K_0&= (L_c-b)^2-(L_c-a)^2,\\ 
		K_1&= a(L_c-b)^2-b(L_c-a)^2, \\
		K_2&= a^2(L_c-b)^2-b^2(L_c-a)^2.
	\end{align*}
	By equating the results of mappings $\sigma$ and $S$ with respect to mean $L_c$ of a mean $M$ and employing formulas \eqref{sigma-explicit} and \eqref{SMc-explicit-k}, we obtain
	\bes
		\frac{L_c \big((1+2c)M- 2(1+c)L_c\big)}{2cM-(1+2c)L_c}
			= \frac{K_1M-K_2}{K_0 M -K_1}
	\ees
	which needs to be proved.
	We calculate
	\bes
		L_c\left[ 2(1+c)L_c-(1+2c)M\right](K_0 M -K_1)
			=\left[ (1+2c)L_c-2cM\right](K_1M-K_2).
	\ees
	Grouping by the powers of $M$ yields
	\begin{align}\label{M2-quadratic}
		\left[ M_0(1+2c)K_0-2cK_1\right] M^2
			&+2\left[ K_2 c-(1+c)L_c^2K_0 \right] M \notag \\
			&+L_c\left[2 (1+c) L_cK_1-(1+2c)K_2\right] = 0.
	\end{align}
	Now we simplify each coefficient by the powers of $M$.
	First,
	\begin{align*}
		M_0(1 &+2c)K_0-2 c K_1=\notag\\
			&= M_0(1+2c)\left[(L_c-b)^2-(L_c-a)^2\right] 
			  -2c\left[ a(L_c-b)^2-b(L_c-a)^2 \right]\notag\\
			&= (a-b)\left[ 2(1+c)L_c^2-(a+b)(1+2c)L_c+2abc \right],
	\end{align*}
	second,
	\begin{align*}
		  cK_2&-(1+c)L_c^2 K_0=\notag\\
			&= c \left[a^2(L_c-b)^2-b^2(L_c-a)^2 \right] 
			  -(1+c)L_c^2\left[ (L_c-b)^2-(L_c-a)^2 \right]\notag\\
			&= -(a-b)L_c \left[ 2(1+c)L_c^2-(a+b)(1+2c)L_c+2abc \right] ,
	\end{align*}
	and third
	\begin{align*}
		 2&(1+c)  L_cK_1-(1+2c)K_2=\notag\\
			&= 2(1+c) L_c \left[a(L_c-b)^2-b(L_c-a)^2 \right]
			 - (1+2c)\left[ a^2(L_c-b)^2-b^2(L_c-a)^2 \right]\notag\\
			&= (a-b)L_c \left[ 2(1+c)L_c^2-(a+b)(1+2c)L_c+2abc \right].
	\end{align*}
	Hence, the equation \eqref{M2-quadratic} factorizes as
	\be\label{M2-quadratic-factorized}
		(a-b)\left[ 2(1+c)L_c^2-(a+b)(1+2c)L_c+2abc \right]
		\left( M^2 -2L_c M +L_c^2 \right) = 0.
	\ee
	Notice that the mean $L_c$ defined in \eqref{theM} is one of the solutions of quadratic equation
	\bes
		 2(1+c)L_c^2-(a+b)(1+2c)L_c+2abc,
	\ees
	and the condition \eqref{M2-quadratic-factorized} is fulfilled which proves the theorem.	
\end{proof}

We will close this section with a conjecture. Based on the analysis in this paper we may conclude the following.

\begin{conjecture}
Symmetric homogeneous mean which has the asymptotic power series expansion and fulfills the requirements of the Open question from \cite{Farhi} necessarily has the same coefficients as mean $L_c$, $c\in [-1,+\infty\rangle$.
\end{conjecture}

\section{Concluding remarks}

Using techniques of asymptotic expansions we were able to compare two symmetries of different origins on the set of mean functions. Finding asymptotic series expansion for both of them, in terms of recursive algorithm for their coefficients, enabled us to carry out the coefficient comparison which resulted wtih obtaining class of means which interpolates between harmonic, geometric and arithmetic mean. Furthermore, varoius extensions of $L_c$, $c\in [-1,+\infty\rangle$, could be observed, such that for all $c \in\bR$ function $L_c$ would be a mean.

Methods presented in this paper may be useful with various problems regarding bivariate means and further. For example, in case of dual means, generalized inverses of means and similar problems where some functional connection is given and especially when the explicit formula for some of the means involved was not known.

\end{document}